\documentclass[12pt,reqno]{amsart}

\usepackage{amssymb}

\usepackage{mathrsfs}

\usepackage{bbm}

\DeclareMathAlphabet{\mathpzc}{OT1}{pzc}{m}{it}

\usepackage[all]{xy}

\entrymodifiers={+!!<0pt,\fontdimen22\textfont2>}

\def\cS{\mathscr{S}}
\def\cT{\mathscr{T}}
\def\cU{\mathscr{U}}
\def\cV{\mathscr{V}}
\def\cW{\mathscr{W}}
\def\cX{\mathscr{X}}

\def\BN{\mathbb{N}}

\def\BZ{\mathbb{Z}}

\def\sA{\mathsf{A}}

\def\sC{\mathsf{C}}
\def\sD{\mathsf{D}}

\def\add{\operatorname{add}}

\def\adots{\mathinner{\mkern1mu\raise1.0pt\vbox{\kern7.0pt\hbox{.}}\mkern2mu\raise4.0pt\hbox{.}\mkern2mu\raise7.0pt\hbox{.}\mkern1mu}}

\def\ann{\operatorname{ann}}

\def\Coker{\operatorname{Coker}}

\def\D{\sD}

\def\Df{\D^{\operatorname{f}}}

\def\End{\operatorname{End}}

\def\Ext{\operatorname{Ext}}

\def\Hom{\operatorname{Hom}}

\def\K0{\operatorname{K}_0}

\def\Ker{\operatorname{Ker}}

\def\mod{\mathsf{mod}}

\def\opp{\operatorname{o}}

\def\rep{\mathsf{rep}}

\newtheorem{Lemma}{Lemma}[section]
\newtheorem{Theorem}[Lemma]{Theorem}
\newtheorem{Proposition}[Lemma]{Proposition}

\theoremstyle{definition}
\newtheorem{Definition}[Lemma]{Definition}
\newtheorem{Setup}[Lemma]{Setup}

\newtheorem{Remark}[Lemma]{Remark}

\begin{document}

\setlength{\parindent}{0pt}
\setlength{\parskip}{7pt}

\title[Cluster and classical tilting]
{On the relation between cluster and classical tilting}

\author{Thorsten Holm}
\address{Institut f\"{u}r Algebra, Zahlentheorie und Diskrete
  Mathematik, Fa\-kul\-t\"{a}t f\"{u}r Ma\-the\-ma\-tik und Physik, Leibniz
  Universit\"{a}t Hannover, Welfengarten 1, 30167 Hannover, Germany}
\email{holm@math.uni-hannover.de}
\urladdr{http://www.iazd.uni-hannover.de/\~{ }tholm}

\author{Peter J\o rgensen}
\address{School of Mathematics and Statistics,
Newcastle University, Newcastle upon Tyne NE1 7RU,
United Kingdom}
\email{peter.jorgensen@ncl.ac.uk}
\urladdr{http://www.staff.ncl.ac.uk/peter.jorgensen}



\keywords{Abelian category, cluster category, cluster tilting, derived
category, Ext, finite global dimension, hereditary abelian category,
maximal $1$-orthogonal subcategory, path algebra, quiver, support
tilting, triangulated category}

\subjclass[2000]{16G70, 18E10, 18E30, 18G99}

\begin{abstract} 

Let $\sD$ be a triangulated ca\-te\-go\-ry with a cluster tilting
subcategory $\cU$.  The quotient category $\sD/\cU$ is abelian;
suppose that it has finite global dimension.

We show that projection from $\sD$ to $\sD/\cU$ sends cluster tilting
sub\-ca\-te\-go\-ri\-es of $\sD$ to support tilting subcategories of
$\sD/\cU$, and that, in turn, support tilting subcategories of
$\sD/\cU$ can be lifted uniquely to maximal $1$-orthogonal
subcategories of $\sD$.
  
\end{abstract}

\maketitle

\setcounter{section}{-1}
\section{Introduction}
\label{sec:introduction}

{\em Classical tilting} is a major subject in the representation
theory of finite dimensional algebras.  According to the historical
remarks in \cite[chp.\ VI]{ASS}, classical tilting theory goes back to
the study of reflection functors by Bernstein, Gelfand, and Ponomarev
in \cite{BGP} and by Auslander, Platzeck, and Reiten in \cite{APR}.
It was later axiomatized by Brenner and Butler in \cite{BrennerButler}
and by Happel and Ringel in \cite{HappelRingel}, and is now one of the
mainstays of representation theory.

Let $Q$ be a finite quiver without loops and cycles and consider the
module category $\mod\, kQ$ of the path algebra $kQ$.  The principal
notion of classical tilting theory is that of a tilting module $T$ in
$\mod\, kQ$.  Such a module satisfies $\Ext_{kQ}^1(T,T) = 0$ and
permits an exact sequence $0 \rightarrow kQ \rightarrow T^0
\rightarrow T^1 \rightarrow 0$ where the $T^i$ are in $\add T$, the
category of direct summands of (finite) direct sums of copies of $T$.
In this situation, $A = \End_{kQ}(T)^{\opp}$ is called a tilted algebra.

{\em Cluster tilting} is a recent, important development in tilting
theory where tilting modules are replaced by so-called cluster tilting
objects; see \cite{BMR2} or the surveys in \cite{BM} and
\cite{RingelTilt}.  These objects live in the cluster ca\-te\-go\-ry
$\sC$ which is the orbit category $\Df(kQ)/\tau^{-1}\Sigma$, where
$\Df(kQ)$ is the finite derived category of $kQ$ while $\tau$ and
$\Sigma$ are the Auslander-Reiten translation and the suspension
functor of $\Df(kQ)$.  The category $\sC$ is triangulated, and a
cluster tilting object $U$ in $\sC$ is defined by satisfying
\[
  u \in \add U \Leftrightarrow \sC(U,\Sigma u) = 0
\]
and
\[
  u \in \add U \Leftrightarrow \sC(u,\Sigma U) = 0
\]
for $u$ in $\sC$.  In this situation, $A = \End_{\sC}(U)^{\opp}$ is
called a cluster tilted algebra.

For any vertex which is a sink or source of $Q$, classical tilting
theory permits the construction of a tilting module whose tilted
algebra has quiver $Q^{\prime}$ given by inverting the arrows of $Q$
incident to the sink or source.  One of the exciting new aspects of
cluster tilting theory is that, in a sense, it permits the extension
of this to arbitrary vertices of $Q$; see \cite[sec.\ 4]{BMR2}. 

{\em A result by Ingalls and Thomas} throws light on the relation
between cluster and classical tilting.  The following precise
statement is part of the main theorem of \cite{IngallsThomas} which
also introduced the concept of support tilting modules.

{\bf Theorem A (Ingalls and Thomas).}
{\em
Let $Q$ be a finite quiver without loops and cycles and let $\sC$ be
the cluster category of type $Q$ over an algebraically closed field
$k$.

Then there is a bijection between the isomorphism classes of basic
cluster til\-ting objects of $\sC$ and the isomorphism classes of basic
support tilting modules in $\mod\, kQ$.
}

As the name suggest, a support tilting module $T$ in $\mod\, kQ$ is a
module which is tilting on its support: It satisfies $\Ext_{kQ}^1(T,T)
= 0$ and is a tilting module for the algebra $kQ/\ann T$ which turns
out to be the path algebra of the support of $T$ in $Q$; see
\cite[prop.\ 2.5 and lem.\ 2.6]{IngallsThomas}.

Ingalls and Thomas prove this theorem by viewing $\mod\, kQ$ as a
sub\-ca\-te\-go\-ry of $\sC$.  There is also a dual viewpoint whereby
$\mod\, kQ$ is a quotient category of $\sC$.  Namely, $kQ$ can be
viewed as a module over itself and hence also as an object of $\sC$.
As such, it is the ``canonical'' cluster tilting object of $\sC$, and
the quotient category $\sC/\add kQ$ is equivalent to $\mod\, kQ$.

The theorem therefore states a relation between the cluster tilting
objects of the triangulated category $\sC$ and the support tilting
objects of the abelian quotient category $\sC/\add kQ$.

{\em The results of this paper} provide similar relations in a general
setup between a triangulated category $\sD$ and the abelian quotient
category $\sD/\cU$, where $\cU$ is a cluster tilting subcategory (see
De\-fi\-ni\-ti\-on \ref{def:cluster_tilting}).  It was proved by
K\"{o}nig and Zhu that $\sD/\cU$ is indeed abelian; see
\cite{KoenigZhu}.

Suppose that $\sD$ satisfies the technical conditions of Setup
\ref{set:blanket} below, and assume that $\sD/\cU$ has finite global
dimension.  Our first main result is the following.

{\bf Theorem B.}
{\em 
Let $\cV$ be a cluster tilting subcategory of $\sD$.  Then the image
$\underline{\cV}$ in $\sD/\cU$ is a support tilting subcategory of
$\sD/\cU$.
}

From the Serre functor $S$ and the suspension functor $\Sigma$ of
$\sD$ can be constructed the autoequivalence $S\Sigma^{-2}$ of $\sD$.
It induces an au\-to\-e\-qui\-va\-len\-ce of $\sD/\cU$ which we also
denote $S\Sigma^{-2}$.  Observe that, related to the notion of a
cluster tilting subcategory, there is the weaker notion of a maximal
$1$-orthogonal subcategory (see Definition \ref{def:cluster_tilting}).
Our second main result is the following.

{\bf Theorem C.}
{\em
Assume that each object of $\sD/\cU$ has finite length.  Let
$\underline{\cW}$ be a support tilting subcategory of $\sD/\cU$ with
$S\Sigma^{-2}\underline{\cW} = \underline{\cW}$.  Then there is a
unique subcategory $\cX$ of $\sD$ which is maximal
$1$-or\-tho\-go\-nal and whose image $\underline{\cX}$ in $\sD/\cU$
satisfies $\underline{\cX} = \underline{\cW}$.  }

The assumption $S\Sigma^{-2}\underline{\cW} = \underline{\cW}$ is
reasonable in the context: In good cases, $\cX$ is not just maximal
$1$-orthogonal but cluster tilting, and then $S\Sigma^{-2}\cX = \cX$
by \cite[prop.\ 4.7.3]{KoenigZhu} which forces
$S\Sigma^{-2}\underline{\cW} = \underline{\cW}$.

It would be nice to dispense with the assumption that $\sD/\cU$ has
finite global dimension, but we presently have no tools for that.  The
proofs of Theorems B and C rely on formulae for $\Ext$ groups in
$\sD/\cU$ in terms of data in $\sD$.  At the moment, we can only prove
such formulae when certain homological dimensions are finite; in
practice, this forces us to assume that $\sD/\cU$ has finite global
dimension.

The paper is organized as follows: Section \ref{sec:push_background}
prepares the ground by proving the mentioned formulae for $\Ext$
groups in $\sD/\cU$ (Proposition \ref{pro:Ext}); this should be of
independent interest.  Section \ref{sec:push} proves Theorem B (see
Theorem \ref{thm:VtoW}), and Section \ref{sec:pull} proves Theorem C
(see Theorem \ref{thm:WtoV}).  Section \ref{sec:examples} considers
some examples: Cluster categories, for which we recover Theorem A,
derived categories of path algebras, and the category of type
$A_{\infty}$ studied in \cite{HJ}.

We would like to mention that, although the work by Ingalls and Thomas
was a main inspiration for this paper, there are also connections to
\cite{FuLiu} and \cite{Smith}.

\begin{Remark}
\label{rmk:subcat}
We will follow a common abuse of terminology by saying that
subcategories are equal when we really mean that they have the same
essential closure, that is, intersect the same set of isomorphism
classes in the ambient category.  For instance, the equation
$S\Sigma^{-2}\underline{\cW} = \underline{\cW}$ in Theorem C must be
read according to this remark.
\end{Remark}

\section{Ext groups in an abelian quotient of a triangulated category}
\label{sec:push_background}

This section gives some background on the abelian quotient category
$\sD/\cU$.  The main item is Proposition \ref{pro:Ext} which, under
certain conditions, gives formulae for the $\Ext$ groups of $\sD/\cU$
in terms of data in the triangulated category $\sD$.

\begin{Setup}
\label{set:blanket}
In the rest of the paper, $k$ is an algebraically closed field and
$\sD$ is a skeletally small $k$-linear triangulated category with
finite dimensional $\Hom$ spaces and split idempotents which has Serre
functor $S$.

By $\cU$ is denoted a cluster tilting subcategory of $\sD$.
\end{Setup}

We refer to \cite[sec.\ I.1]{RVdB} for background on Serre functors,
but wish to recall the following definitions; cf.\ \cite{BMRRT},
\cite{Iyama2}, \cite{Iyama1}, \cite{IyamaYoshino}, and
\cite{KellerReiten2}.

\begin{Definition}
\label{def:cluster_tilting}
A full subcategory $\cV$ of $\sD$ is called {\em maximal
$1$-or\-tho\-go\-nal} if it satisfies
\[
  v \in \cV \Leftrightarrow \sD(\cV,\Sigma v) = 0
\]
and
\[
  v \in \cV \Leftrightarrow \sD(v,\Sigma \cV) = 0.
\]

A maximal $1$-orthogonal subcategory is called {\em cluster tilting}
if it is pre\-co\-ve\-ring and preenveloping.
\end{Definition}

\begin{Remark}
Our distinction between maximal $1$-or\-tho\-go\-nal and cluster
tilting subcategories is not standard, but it is useful for this
paper.

In the definition, recall that $\cV$ is called {\em precovering} if each
object $x$ of $\sD$ has a $\cV$-precover, that is, a morphism $v
\rightarrow x$ with $v$ in $\cV$ through which any other morphism
$v^{\prime} \rightarrow x$ with $v^{\prime}$ in $\cV$ factors.
Dually, $\cV$ is called {\em preenveloping} if each object $x$ of
$\sD$ has a $\cV$-preenvelope, that is, a morphism $x \rightarrow v$
with $v$ in $\cV$ through which any other morphism $x \rightarrow
v^{\prime}$ with $v^{\prime}$ in $\cV$ factors.
\end{Remark}

\begin{Remark}
\label{rmk:DmodU}
The quotient category $\sD/\cU$ has the same objects as $\sD$, and its
$\Hom$ spaces are obtained from those of $\sD$ upon dividing by the
morphisms which factor through an object of $\cU$.  The projection
functor $\sD \rightarrow \sD/\cU$ will be denoted by $x \mapsto
\underline{x}$.  The space of morphisms $x \rightarrow y$ which factor
through an object of $\cU$ will be denoted $\cU(x,y)$, so
\[
  (\sD/\cU)(\underline{x},\underline{y}) = \sD(x,y)/\cU(x,y).
\]

The category $\sD$ is Krull-Schmidt by \cite[p.\ 52]{RingTame}.  By
\cite[lem.\ 2.1]{KoenigZhu} so is  $\sD/\cU$, and the
projection functor $\sD \rightarrow \sD/\cU$ induces a bijective
correspondence between the isomorphism classes of indecomposable
objects of $\sD/\cU$ and the isomorphism classes of indecomposable
objects of $\sD$ which are outside $\cU$.


By \cite[thm.\ 3.3, prop.\ 4.2, and thm.\ 4.3]{KoenigZhu}, the
category $\sD/\cU$ is abelian with enough projective and injective
objects.  Its projectives are the objects isomorphic to objects in
$\underline{\Sigma^{-1}\cU}$ and its injectives are the objects
isomorphic to objects in $\underline{\Sigma\cU}$.

By \cite[cor.\ 4.4]{KoenigZhu}, there is an equivalence $\sD/\cU
\simeq \mod\, \Sigma^{-1}\cU$.  The right hand side is clearly
equivalent to $\mod\, \cU$, so we have $\sD/\cU \simeq \mod\, \cU$.

Let $\Sigma^{-1}u$ be in $\Sigma^{-1}\cU$ and $x$ in $\sD$.  It is a
useful observation that since we have $\sD(\Sigma^{-1}\cU,\cU) = 0$,
there is an isomorphism
\[
  \sD(\Sigma^{-1}u,x)
  \cong (\sD/\cU)(\underline{\Sigma^{-1}u},\underline{x}).
\]

Let $x \rightarrow y \rightarrow z \rightarrow$ be a distinguished
triangle in $\sD$.  The composition of two consecutive morphisms in a
distinguished triangle is zero and remains so on projecting to
$\sD/\cU$, so there is an induced sequence $\underline{x} \rightarrow
\underline{y} \rightarrow \underline{z}$ in $\sD/\cU$.  This is an
exact sequence.  To see so, it is enough to check that it becomes
exact under the functor $(\sD/\cU)(\underline{p},-)$ when
$\underline{p}$ is projective in $\sD/\cU$.  We can assume $p =
\Sigma^{-1}u$ for a $u$ in $\cU$, so we must show that
\[
  (\sD/\cU)(\underline{\Sigma^{-1}u},\underline{x})
  \rightarrow (\sD/\cU)(\underline{\Sigma^{-1}u},\underline{y})
  \rightarrow (\sD/\cU)(\underline{\Sigma^{-1}u},\underline{z})
\]
is exact.  By the above this is just
\[
  \sD(\Sigma^{-1}u,x)
  \rightarrow \sD(\Sigma^{-1}u,y)
  \rightarrow \sD(\Sigma^{-1}u,z)
\]
which is indeed exact.

By repeatedly ``turning'' the distinguished triangle, it is possible
to obtain a long sequence in $\sD$ in which each four term part is a
distinguished triangle.  This induces a long exact sequence in
$\sD/\cU$.  

By \cite[prop.\ 4.7.3]{KoenigZhu}, the autoequivalence $S\Sigma^{-2}$
of $\sD$ satisfies $S\Sigma^{-2}\cU = \cU$.  Hence $S\Sigma^{-2}$
induces an autoequivalence of $\sD/\cU$ which, by abuse of notation,
will also be denoted $S\Sigma^{-2}$.
\end{Remark}

In the following result, recall that $\cU(x,\Sigma y)$ is the space of
morphisms $x \rightarrow \Sigma y$ in $\sD$ which factor through an
object from $\cU$.

\begin{Proposition}
\label{pro:Ext}
Let $x$ and $y$ be in $\sD$.
\begin{enumerate}
  
  \item If $x$ has no direct summands from $\cU$ and $\underline{x}$ has
  finite projective dimension in $\sD/\cU$, then
\[
  \Ext_{\sD/\cU}^1(\underline{x},\underline{y}) \cong \cU(x,\Sigma y).
\]

  \item If $y$ has no direct summands from $\cU$ and $\underline{y}$ has
  finite injective dimension in $\sD/\cU$, then
\[
  \Ext_{\sD/\cU}^1(\underline{x},\underline{y}) \cong \cU(\Sigma^{-1}x,y).
\]

\end{enumerate}
\end{Proposition}

\begin{proof}  
We will only prove (i) since (ii) can be established by the dual
argument.

Since $\underline{x}$ has finite projective dimension in $\sD/\cU$,
its projective dimension is at most one, see \cite[thm.\ 
4.3]{KoenigZhu} and \cite[2.1, cor.]{KellerReiten2}.

By \cite[lem.\ 3.2.1]{KoenigZhu}, there is a distinguished triangle
\[
  \Sigma^{-1}u_1
  \stackrel{\alpha}{\longrightarrow} \Sigma^{-1}u_0
  \longrightarrow x
  \longrightarrow
\]
in $\sD$ where the $u_i$ are in $\cU$.  Turning the triangle gives a
sequence
\begin{equation}
\label{equ:h}
  \Sigma^{-2}u_0
  \stackrel{\gamma}{\longrightarrow} \Sigma^{-1}x
  \stackrel{\beta}{\longrightarrow} \Sigma^{-1}u_1
  \stackrel{\alpha}{\longrightarrow} \Sigma^{-1}u_0
  \longrightarrow x
  \longrightarrow u_1
\end{equation}
which by Remark \ref{rmk:DmodU} induces a long exact sequence in
$\sD/\cU$,
\[
  \underline{\Sigma^{-1}x}
  \stackrel{\underline{\beta}}{\longrightarrow} \underline{\Sigma^{-1}u_1}
  \stackrel{\underline{\alpha}}{\longrightarrow} \underline{\Sigma^{-1}u_0}
  \longrightarrow \underline{x}
  \longrightarrow \underline{u_1}.
\]
In $\sD/\cU$ the object $\underline{u_1}$ is isomorphic to $0$, so the
penultimate morphism is an epimorphism onto $\underline{x}$.  The
object $\underline{\Sigma^{-1}u_0}$ is projective and $\underline{x}$
has projective dimension at most one, so the image $\underline{p}$ of
$\underline{\alpha}$ is projective and so $\underline{\alpha}$ viewed
as a morphism to $\underline{p}$ is a split epimorphism.  Hence the
kernel $\underline{q}$ of $\underline{\alpha}$ is a direct summand of
$\underline{\Sigma^{-1}u_1}$, and since $\underline{\Sigma^{-1}u_1}$
is projective so is $\underline{q}$.  But $\underline{q}$ is also
the image of $\underline{\beta}$, and so $\underline{\beta}$
viewed as a morphism to $\underline{q}$ is a split epimorphism.  Hence
the kernel $\underline{z}$ of $\underline{\beta}$ is a direct summand
of $\underline{\Sigma^{-1}x}$.

Putting together this information, the exact sequence is isomorphic to
\[
  \xymatrix @C+2.5pc {
    \underline{z} \oplus \underline{q}
      \ar[r]^{\left(\begin{array}{cc} {\scriptstyle 0} & {\scriptstyle 1} \\ {\scriptstyle 0} & {\scriptstyle 0} \end{array}\right)} &
    \underline{q} \oplus \underline{p}
      \ar[r]^{\left(\begin{array}{cc} {\scriptstyle 0} & \mbox{\tiny mono} \end{array}\right)} &
    \underline{\Sigma^{-1}u_0}
      \ar[r] &
    \underline{x}
      \ar[r] &
    0.
           }
\]
In particular we have $\underline{\Sigma^{-1}x} \cong \underline{z}
\oplus \underline{q}$ in $\sD/\cU$.  But $x$ has no direct summands
from $\cU$ so $\underline{\Sigma^{-1}x}$ has no direct summands from
$\underline{\Sigma^{-1}\cU}$; that is, $\underline{\Sigma^{-1}x}$ has
no projective direct summands so $\underline{q} \cong 0$.  Hence the
exact sequence is isomorphic to
\[
  \underline{z}
  \stackrel{0}{\longrightarrow} \underline{p}
  \stackrel{\underline{\alpha}}{\longrightarrow} \underline{\Sigma^{-1}u_0}
  \longrightarrow \underline{x}
  \longrightarrow 0.
\]

It follows that
\begin{align*}
  \Ext_{\sD/\cU}^1(\underline{x},\underline{y})
  & \cong \Coker\:(\sD/\cU)(\underline{\alpha},\underline{y}) \\
  & \stackrel{\rm (a)}{=} \Coker \sD(\alpha,y) \\
  & \stackrel{\rm (b)}{=} \Ker \sD(\gamma,y) \\
  & = (*)
\end{align*}
where (a) is by Remark \ref{rmk:DmodU} because $\alpha$ is a morphism
in $\Sigma^{-1}\cU$ and (b) is by equation \eqref{equ:h}.  But the
kernel $(*)$ consists of the morphisms $\Sigma^{-1}x \rightarrow y$
which factor through $\beta$, and it is easy to check that these are
precisely the morphisms which factor through some object of
$\Sigma^{-1}\cU$ whence
\[
  (*) = (\Sigma^{-1}\cU)(\Sigma^{-1}x,y)
  \cong \cU(x,\Sigma y).
\]
\end{proof}

\section{Projecting a cluster tilting subcategory}
\label{sec:push}

This section proves Theorem B from the Introduction; see Theorem
\ref{thm:VtoW}.

The following is a straightforward abstraction of the notion of
support tilting modules from \cite{IngallsThomas}.

\begin{Definition}
\label{def:support_tilting}
To say that $\cS$ is a {\em support tilting subcategory} of an abelian 
category $\sA$ means that $\cS$ is a full subcategory which
\begin{itemize}
  \item  is closed under (finite) direct sums and direct summands;
  \item  is precovering and preenveloping;
  \item  satisfies $\Ext_{\sA}^2(\cS,-) = 0$;
  \item  satisfies $\Ext_{\sA}^1(\cS,\cS) = 0$;
  \item  satisfies that if $y$ is a subquotient of an object from
    $\cS$ for which we have $\Ext_{\sA}^1(\cS,y) = 0$, then $y$ is a
    quotient of an object from $\cS$.
\end{itemize}
\end{Definition}

%

\begin{Theorem}
\label{thm:VtoW}
Assume that $\sD/\cU$ has finite global dimension.

Let $\cV$ be a cluster tilting subcategory of $\sD$.  Then the image
$\underline{\cV}$ is a support tilting subcategory of $\sD/\cU$.
\end{Theorem}

\begin{proof}
Since $\cV$ is cluster tilting, it is closed under direct sums and
direct summands, as follows from Definition \ref{def:cluster_tilting}.
Hence $\underline{\cV}$ is closed under direct sums and direct
summands.

Moreover, $\underline{\cV}$-precovers and
$\underline{\cV}$-preenvelopes are induced by $\cV$-pre\-co\-vers and
$\cV$-preenvelopes, so $\underline{\cV}$ is precovering and
preenveloping.


The objects of $\underline{\cV}$ have finite projective dimension
since $\sD/\cU$ has finite global dimension, so each object of
$\underline{\cV}$ has projective dimension at most one by
\cite[thm.\ 4.3]{KoenigZhu} and \cite[2.1, cor.]{KellerReiten2}.
Hence the condition $\Ext_{\sD/\cU}^2(\underline{\cV},-) = 0$ is
satisfied. 

For $\underline{v}$ and $\underline{v^{\prime}}$ in $\underline{\cV}$,
let us prove $\Ext_{\sD/\cU}^1(\underline{v},\underline{v^{\prime}}) =
0$.  We can discard any direct summands of $v$ which are in $\cU$
since they do not make any difference to the isomorphism class of
$\underline{v}$.  But $\underline{v}$ has finite projective dimension
in $\sD/\cU$ since that category has finite global dimension, so
$\Ext_{\sD/\cU}^1(\underline{v},\underline{v^{\prime}}) \cong
\cU(v,\Sigma v^{\prime})$ by Proposition \ref{pro:Ext}(i), and here
the right hand side is zero since it is a subspace of $\sD(v,\Sigma
v^{\prime})$ which is zero because $\cV$ is cluster tilting.

Finally, let $\underline{y}$ be a subquotient of $\underline{v}$ in
$\sD/\cU$ where $\underline{v}$ is in $\underline{\cV}$, and suppose
$\Ext_{\sD/\cU}^1(\underline{\cV},\underline{y}) = 0$.  Let us prove
that $\underline{y}$ is a quotient of an object from
$\underline{\cV}$.

We can discard any direct summands of $y$ which are in $\cU$.
Moreover, $\underline{y}$ has finite injective dimension because
$\sD/\cU$ has finite global dimension.  It follows by Proposition
\ref{pro:Ext}(ii) that
\begin{equation}
\label{equ:a}
  \cU(\Sigma^{-1}v^{\prime},y)
    \cong \Ext_{\sD/\cU}^1(\underline{v^{\prime}},\underline{y}) = 0
\end{equation}
for each $v^{\prime}$ in $\cV$.

%

For $\underline{y}$ to be a subquotient of $\underline{v}$ means that
we have an epimorphism and a monomorphism $\underline{v}
\twoheadrightarrow \underline{t} \hookleftarrow \underline{y}$.  Lift
these two morphisms to $\sD$ and complete to distinguished triangles.
Since the morphisms in $\sD/\cU$ are, respectively, an epimorphism and
a monomorphism, \cite[thm.\ 2.3]{KoenigZhu} implies that the other
morphisms in the distinguished triangles factor as follows,
\[
  \xymatrix{
    v \ar[rr]^{\sigma} & & t \ar[rr]^{\tau} \ar[dr] & & c \ar[rr] & & {} \\
    & & & u \ar[ur] \\
           }
\]
and
\[
  \xymatrix{
    k \ar[rr]^{\kappa} \ar[dr] & & y \ar[rr]^{\gamma} & & t \ar[rr] & & , \\
    & u^{\prime} \ar[ur]_{\mu^{\prime}} \\
           }
\]
with $u$ and $u^{\prime}$ in $\cU$.

For $v^{\prime}$ in $\cV$, the image of
\[
  \sD(\Sigma^{-1}v^{\prime},\mu^{\prime})
    : \sD(\Sigma^{-1}v^{\prime},u^{\prime})
      \rightarrow \sD(\Sigma^{-1}v^{\prime},y)
\]
is a subset of $\cU(\Sigma^{-1}v^{\prime},y)$ which is zero by
equation \eqref{equ:a}.  So we have
$\sD(\Sigma^{-1}v^{\prime},\mu^{\prime}) = 0$ and by Serre duality
$\sD(\mu^{\prime},S\Sigma^{-1}v^{\prime}) = 0$ where $S$ is the Serre
functor of $\sD$.  But \cite[prop.\ 4.7]{KoenigZhu} implies that
$S\Sigma^{-1}\cV = \Sigma\cV$, so it follows that
\begin{equation}
\label{equ:b}
  \sD(\mu^{\prime},\Sigma v^{\prime\prime}) = 0
\end{equation}
for each $v^{\prime\prime}$ in $\cV$.

Now use \cite[lem.\ 3.2.1]{KoenigZhu} to construct a distinguished
triangle in $\sD$,
\[
  \xymatrix{
    v^{\prime} \ar[rr]^{\sigma^{\prime}} & & y \ar[rr]^{\beta} & & \Sigma v^{\prime\prime} \ar[rr] & & {},
           }
\]
with $v^{\prime}$ and $v^{\prime\prime}$ in $\cV$.  Combining the
three distinguished triangles we have constructed gives the solid
arrows in the following commutative diagram,
\[
  \vcenter{
  \xymatrix{
    & & k \ar[dd]^{\kappa} \ar[dl] \\
    & u^{\prime} \ar[dr]_{\mu^{\prime}} \\
    v^{\prime} \ar[rr]_{\sigma^{\prime}} & & y \ar[dd]_{\gamma} \ar[rr]_{\beta} & & \Sigma v^{\prime\prime} \ar[rr] & & {}\\
    {} \\
    v \ar[rr]_{\sigma} & & t \ar[rr]_{\tau} \ar[dd] \ar[dr] \ar@{-->}[uurr]_{\theta} & & c \ar[rr] \ar@{-->}[uu]_{\chi}& & {} \\
    & & & u \ar[ur] \\
    & & {} \\  
           }
          }\;.
\]
Here $\sD(\mu^{\prime},\Sigma v^{\prime\prime}) = 0$ by equation
\eqref{equ:b}, so in particular $\beta\mu^{\prime} = 0$.  It follows that
$\beta\kappa = 0$.  Hence $\theta$ exists with $\theta\gamma = \beta$,
but $\theta\sigma = 0$ since $\sD(\cV,\Sigma\cV) = 0$ so finally,
$\chi$ exists with $\chi\tau = \theta$.

That is, $\beta = \chi\tau\gamma$, but $\tau$ factors through $u$ so
$\beta$ also factors through $u$.  By \cite[thm.\ 2.3]{KoenigZhu}, it
follows that $\underline{\sigma^{\prime}}$ is an epimorphism in
$\sD/\cU$, so $\underline{y}$ is a quotient of the object
$\underline{v^{\prime}}$ from $\underline{\cV}$.
\end{proof}

\section{Lifting a support tilting subcategory}
\label{sec:pull}

This section proves Theorem C from the Introduction; see Theorem
\ref{thm:WtoV}.

\begin{Remark}
\label{rmk:WtoV}
In this section, we will often consider a special way of lifting a
full subcategory from $\sD/\cU$ to $\sD$.

Namely, consider a full subcategory of $\sD/\cU$ which is closed under
direct sums and direct summands.  We can (and will) assume that it has
the form $\underline{\cW}$ where $\cW$ is a full subcategory of $\sD$
which is closed under direct sums and direct summands and consists of
objects without direct summands from $\cU$.  Note that there is a
bijective correspondence between isomorphism classes of indecomposable
objects of $\cW$ and of $\underline{\cW}$.

A lifting of $\underline{\cW}$ to $\sD$ is a subcategory $\cX$ of
$\sD$ with $\underline{\cX} = \underline{\cW}$.  Obviously, $\cW$ is a
lifting of $\underline{\cW}$ to $\sD$, and any other lifting which is
a full subcategory closed under direct sums and direct summands has
the form
\[
  \cX = \add(\cW \cup \cT)
\]
where $\cT$ is contained in $\cU$.

We wish to consider the specific choice
\[
  \cT = \{\, u \in \cU \,|\, \sD(\cW,\Sigma u) = 0 \,\}
\]
since the resulting $\cX$ has the following property: If it is
possible to lift $\underline{\cW}$ to a maximal $1$-orthogonal
subcategory $\cX^{\prime}$ of $\sD$, then $\cX^{\prime} = \cX$.

Namely, suppose that $\cX^{\prime}$ exists.  Since $\cX^{\prime}$ is a
lifting of $\underline{\cW}$, we have $\cX^{\prime} = \add(\cW \cup
\cT^{\prime})$ for a $\cT^{\prime}$ which is contained in $\cU$.  We
can take $\cT^{\prime}$ to be closed under direct sums and direct
summands.

On one hand, if an indecomposable $u$ from $\cU$ has $\sD(\cW,\Sigma
u) = 0$, then $\sD(\cX^{\prime},\Sigma u) = 0$ since $\cT^{\prime}$ is
contained in $\cU$, and consequently $u$ is in $\cX^{\prime}$ and so
must be in $\cT^{\prime}$.  On the other hand, if an indecomposable
$u$ from $\cU$ has $\sD(\cW,\Sigma u) \neq 0$, then
$\sD(\cX^{\prime},\Sigma u) \neq 0$, and consequently $u$ is not in
$\cX^{\prime}$ and so cannot be in $\cT^{\prime}$.  Hence
$\cT^{\prime} = \cT$ and $\cX^{\prime} = \cX$.
\end{Remark}

\begin{Lemma}
\label{lem:WtoV}
Let $\underline{\cW}$ and $\cW$ be as in Remark \ref{rmk:WtoV}, and
assume that each object of $\underline{\cW}$ has finite projective
dimension, that $\Ext_{\sD/\cU}^1(\underline{\cW},\underline{\cW}) =
0$, and that $S\Sigma^{-2}\underline{\cW} = \underline{\cW}$.  Then
$\sD(\cW,\Sigma\cW) = 0$.
\end{Lemma}

\begin{proof}
Let $w$ and $w^{\prime}$ be objects of $\cW$.  Since objects of $\cW$
have no direct summands from $\cU$, the condition
$S\Sigma^{-2}\underline{\cW} = \underline{\cW}$ implies
$S\Sigma^{-2}\cW = \cW$ whence
\begin{equation}
\label{equ:bb}
  \Sigma^2 w^{\prime} \cong S\widetilde{w}
\end{equation}
for an object $\widetilde{w}$ in $\cW$.

By \cite[lem.\ 3.2.1]{KoenigZhu} there is a distinguished triangle
\[
  u^1 \rightarrow \Sigma w \rightarrow \Sigma u^0 \rightarrow
\]
with the $u^i$ in $\cU$.  This induces an exact sequence
\[
  \sD(\widetilde{w},u^1)
  \stackrel{\alpha}{\longrightarrow} \sD(\widetilde{w},\Sigma w)
  \stackrel{\beta}{\longrightarrow} \sD(\widetilde{w},\Sigma u^0),
\]
and it is easy to check that the image of $\alpha$ is
$\cU(\widetilde{w},\Sigma w)$ which by Proposition \ref{pro:Ext}(i) is
$\Ext_{\sD/\cU}^1(\underline{\widetilde{w}},\underline{w})$ since
$\widetilde{w}$ has no direct summands from $\cU$ and since
$\underline{\widetilde{w}}$ has finite projective dimension because it
is in $\underline{\cW}$.  By assumption this $\Ext$ is zero, so
$\beta$ is injective.

Using the Serre functor $S$ and $k$-linear duality $(-)^{\vee} =
\Hom_k(-,k)$ along with equation \eqref{equ:bb}, we can rewrite $\beta$
as follows,
\[
  \xymatrix{
    \sD(\widetilde{w},\Sigma w) \ar[r]^{\beta} \ar[d]_{\cong} & \sD(\widetilde{w},\Sigma u^0) \ar[d]^{\cong} \\
    \sD(\Sigma w,S\widetilde{w})^{\vee} \ar[r] \ar[d]_{\cong} & \sD(\Sigma u^0,S\widetilde{w})^{\vee} \ar[d]^{\cong}\\
    \sD(\Sigma w,\Sigma^2 w^{\prime})^{\vee} \ar[r] \ar[d]_{\cong} & \sD(\Sigma u^0,\Sigma^2 w^{\prime})^{\vee} \ar[d]^{\cong}\\
    \sD(w,\Sigma w^{\prime})^{\vee} \ar[r] & \sD(u^0,\Sigma w^{\prime})^{\vee}\lefteqn{,} \\
           }
\]
and since these maps are injective, the dual $\sD(u^0,\Sigma
w^{\prime}) \rightarrow \sD(w,\Sigma w^{\prime})$ of the last map is
surjective.  It is easy to see that the image of this map is
$\cU(w,\Sigma w^{\prime})$, so we have
\[
  \sD(w,\Sigma w^{\prime}) = \cU(w,\Sigma w^{\prime}) = (*).
\]
But
\[
  (*) \cong \Ext_{\sD/\cU}^1(\underline{w},\underline{w^{\prime}})
\]
by Proposition \ref{pro:Ext}(i).  By assumption this $\Ext$ is zero,
so $\sD(w,\Sigma w^{\prime}) = 0$ as claimed.
\end{proof}


\begin{Lemma}
\label{lem:subquotient}
Assume that $\sD/\cU$ has finite global dimension and that each object
of $\sD/\cU$ has finite length.

Let $\underline{\cW}$ be a full subcategory of $\sD/\cU$ which is
closed under direct sums and direct summands, and assume
$\Ext_{\sD/\cU}^1(\underline{\cW},\underline{\cW}) = 0$.

Let $\underline{a}$ be an object of $\sD/\cU$ for which the following
implication holds when $\underline{i}$ is an injective object of
$\sD/\cU$: 
\[
  (\sD/\cU)(\underline{a},\underline{i}) \neq 0 \;\;\Rightarrow
  \mbox{ there is a $\underline{w}$ in $\underline{\cW}$ such that }
  (\sD/\cU)(\underline{w},\underline{i}) \neq 0.
\]
Then $\underline{a}$ is a subquotient in $\sD/\cU$ of an object from
$\underline{\cW}$.
\end{Lemma}

\begin{proof}
It is easy to check that, since $\sD/\cU$ has enough injectives and
all its objects have finite length, $\sD/\cU$ has injective envelopes.
Let $e(\underline{t})$ be the injective envelope of a simple object
$\underline{t}$.  It is also easy to check that $\underline{t}$
appears in the composition series of an object $\underline{a}$ if and
only if $(\sD/\cU)(\underline{a},e(\underline{t})) \neq 0$.

Now let the simple object $\underline{t}$ be in the composition series
of the object $\underline{a}$.  Then
$(\sD/\cU)(\underline{a},e(\underline{t})) \neq 0$ whence, by the
assumption of the lemma, $(\sD/\cU)(\underline{w},e(\underline{t})) \neq
0$ for some $\underline{w}$ in $\underline{\cW}$.  This in turn means
that $\underline{t}$ appears in the composition series of
$\underline{w}$, so $\underline{t}$ is a subquotient of an object of
$\underline{\cW}$.

But $\underline{a}$ is a successive extension of the simple objects in
its composition series, so $\underline{a}$ is a successive extension
of subquotients of objects of $\underline{\cW}$.  The method used in
the proof of \cite[lem.\ 2.4]{IngallsThomas} shows that the class of
subquotients of objects from $\underline{\cW}$ is closed under
extensions, so it follows that $\underline{a}$ is a subquotient of an
object from $\underline{\cW}$.
\end{proof}

\begin{Theorem}
\label{thm:WtoV}
Assume that $\sD/\cU$ has finite global dimension and that each object
of $\sD/\cU$ has finite length.

Let $\underline{\cW}$ be a support tilting subcategory of $\sD/\cU$
with $S\Sigma^{-2}\underline{\cW} = \underline{\cW}$.  Then the
category $\cX$ from Remark \ref{rmk:WtoV} is the unique maximal
$1$-or\-tho\-go\-nal subcategory of $\sD$ which is a lifting of
$\underline{\cW}$.
\end{Theorem}

\begin{proof}
Remark \ref{rmk:WtoV} says that $\cX$ is a lifting of
$\underline{\cW}$ to $\sD$, and that if there is a maximal
$1$-orthogonal lifting $\cX^{\prime}$ then $\cX^{\prime} = \cX$.  So
we just need to show that $\cX$ is indeed maximal $1$-orthogonal; that
is,
\begin{align*}
  x \in \cX & \Leftrightarrow \sD(\cX,\Sigma x) = 0, \\
  x \in \cX & \Leftrightarrow \sD(x,\Sigma \cX) = 0.
\end{align*}

Since the objects of $\cW$ have no direct summands from $\cU$, the
condition $S\Sigma^{-2}\underline{\cW} = \underline{\cW}$ implies
$S\Sigma^{-2}\cW = \cW$.

The implications $\Rightarrow$.  It is enough to show $\sD(x,\Sigma y)
= 0$ for indecomposable objects $x$ and $y$ of $\cX$.  Recall the
construction from Remark \ref{rmk:WtoV}; in particular $\cX = \add(\cW
\cup \cT)$ so we may assume that each of $x$ and $y$ is in $\cW$ or
$\cT$.

If $x$ and $y$ are in $\cW$, then Lemma \ref{lem:WtoV} gives
$\sD(x,\Sigma y) = 0$.

If $x$ and $y$ are in $\cT$, then they are in particular in $\cU$
whence $\sD(x,\Sigma y) = 0$.

If $x$ is in $\cW$ and $y$ is in $\cT$, then $\sD(x,\Sigma y) = 0$ by
the definition of $\cT$ in Remark \ref{rmk:WtoV}.

Finally, if $x$ is in $\cT$ and $y$ is in $\cW$, then $y \cong
S\Sigma^{-2}w$ for a $w$ in $\cW$ since $S\Sigma^{-2}\cW = \cW$.
So
\begin{align*}
  \sD(x,\Sigma y)
  & \cong \sD(x,\Sigma S\Sigma^{-2}w) \\
  & \cong \sD(x,S\Sigma^{-1}w) \\
  & \cong \sD(\Sigma^{-1}w,x)^{\vee} \\
  & \cong \sD(w,\Sigma x)^{\vee},
\end{align*}
and the right hand side is zero by the definition of $\cT$.

The implications $\Leftarrow$.  We know $S\Sigma^{-2}\cW = \cW$, and
$S\Sigma^{-2}\cU = \cU$ by \cite[prop.\ 4.7]{KoenigZhu}.  It follows
that $S\Sigma^{-2}\cT = \cT$, and hence $S\Sigma^{-2}\cX = \cX$.  So
\begin{align}
\label{equ:d}
  \sD(x,\Sigma \cX) = 0
  & \Leftrightarrow \sD(x,\Sigma S\Sigma^{-2}\cX) = 0 \\
\nonumber
  & \Leftrightarrow \sD(x,S\Sigma^{-1}\cX) = 0 \\
\nonumber
  & \Leftrightarrow \sD(\Sigma^{-1}\cX,x)^{\vee} = 0 \\
\nonumber
  & \Leftrightarrow \sD(\cX,\Sigma x)^{\vee} = 0 \\
\nonumber
  & \Leftrightarrow \sD(\cX,\Sigma x) = 0,
\end{align}
and it is sufficient to prove the first implication $\Leftarrow$.  So
let $x$ be an indecomposable object of $\sD$ with $\sD(\cX,\Sigma x) =
0$; in particular
\begin{equation}
\label{equ:aa}
  \sD(\cW,\Sigma x) = 0.
\end{equation}

If $x$ is in $\cU$ then \eqref{equ:aa} says that $x$ is in $\cT$ and
so $x$ is in $\cX$.

Suppose that $x$ is not in $\cU$; then $\underline{x}$ is non-zero and
indecomposable in $\sD/\cU$.  By Proposition \ref{pro:Ext}(i),
equation \eqref{equ:aa} implies
$\Ext_{\sD/\cU}^1(\underline{\cW},\underline{x}) = 0$ since the
objects of $\cW$ have no direct summands from $\cU$ and since the
objects of $\underline{\cW}$ have finite projective dimension.

Let $\underline{i}$ be an injective object of $\sD/\cU$ and suppose
that $(\sD/\cU)(\underline{x},\underline{i}) \neq 0$.  Then $\sD(x,i)
\neq 0$.  By \cite[prop.\ 4.2]{KoenigZhu}, we can suppose $i = \Sigma
u$ for a $u$ in $\cU$.  So we have $\sD(x,\Sigma u) \neq 0$, and since
$\sD(x,\Sigma\cX) = 0$ by equation \eqref{equ:d}, this forces $u$ to
have a direct summand in $\cU$ outside $\cT$.  Then there exists a $w$
in $\cW$ with $\sD(w,\Sigma u) \neq 0$, but this implies
$(\sD/\cU)(\underline{w},\underline{\Sigma u}) \neq 0$, that is,
$(\sD/\cU)(\underline{w},\underline{i}) \neq 0$.  We have shown
\[
  (\sD/\cU)(\underline{x},\underline{i}) \neq 0
  \; \Rightarrow \mbox{ there is a $\underline{w}$ in $\underline{\cW}$ such that }
  (\sD/\cU)(\underline{w},\underline{i}) \neq 0.
\]
It follows from Lemma \ref{lem:subquotient} that $\underline{x}$ is a
subquotient of an object from $\underline{\cW}$.  But we already know
$\Ext_{\sD/\cU}^1(\underline{\cW},\underline{x}) = 0$, and since
$\underline{\cW}$ is support tilting it follows that $\underline{x}$
is a quotient of an object from $\underline{\cW}$.

Consequently, each $\underline{\cW}$-precover of $\underline{x}$ is an
epimorphism.  Pick a precover and complete to a short exact sequence,
\begin{equation}
\label{equ:c}
  0 \rightarrow \underline{k}
    \rightarrow \underline{w}
    \rightarrow \underline{x}
    \rightarrow 0.
\end{equation}
The long exact $\Ext$ sequence implies that
$\Ext_{\sD/\cU}^1(\underline{\cW},\underline{k}) = 0$, 
so since $\underline{k}$ is a subobject and in particular a
subquotient of $\underline{w}$, the support tilting property of
$\underline{\cW}$ shows that $\underline{k}$ is a quotient of an
object from $\underline{\cW}$,
\[
  0 \rightarrow \underline{k^{\prime}}
    \rightarrow \underline{w^{\prime}}
    \rightarrow \underline{k}
    \rightarrow 0.
\]

Now, our assumption is that $\sD(\cX,\Sigma x) = 0$, and by equation
\eqref{equ:d} this implies $\sD(x,\Sigma \cX) = 0$ and in particular
$\sD(x,\Sigma \cW) = 0$.  By Proposition \ref{pro:Ext}(i), it follows
that $\Ext_{\sD/\cU}^1(\underline{x},\underline{\cW}) = 0$ because $x$
has no direct summands from $\cU$ while $\underline{x}$ has finite
projective dimension since $\sD/\cU$ has finite global dimension.  So
in particular $\Ext_{\sD/\cU}^1(\underline{x},\underline{w^{\prime}})
= 0$, and since the projective dimension of $\underline{x}$ is at most
one by \cite[thm.\ 4.3]{KoenigZhu} and \cite[2.1,
cor.]{KellerReiten2}, the long exact $\Ext$ sequence then implies
$\Ext_{\sD/\cU}^1(\underline{x},\underline{k}) = 0$.
Hence the exact sequence \eqref{equ:c} is split, and since
$\underline{w}$ is in $\underline{\cW}$ it follows that
$\underline{x}$ is isomorphic to an object of $\underline{\cW}$.  But
then the indecomposable $x$ is isomorphic to an object of $\cW$ since
$x$ is outside $\cU$, and hence $x$ is in $\cX$.
\end{proof}

\begin{Remark}
In the following proposition and in Section \ref{sec:examples} we will
consider a bijective correspondence between cluster tilting
subcategories and support tilting subcategories.

Tacitly, the correspondence is in fact between equivalence classes of
such subcategories, the equivalence relation being that subcategories
with the same essential closure are equivalent; cp.\ Remark
\ref{rmk:subcat}.
\end{Remark}

\begin{Proposition}
\label{pro:WtoV}
Assume that $\sD/\cU$ has finite global dimension and that each object
of $\sD/\cU$ has finite length.

Suppose that the following condition is satisfied: If
$\underline{\cW}$ is a support tilting subcategory of $\sD/\cU$ with
$S\Sigma^{-2}\underline{\cW} = \underline{\cW}$, then the maximal
$1$-orthogonal subcategory $\cX$ of Remark
\ref{rmk:WtoV} and Theorem \ref{thm:WtoV} is pre\-co\-ve\-ring and
preenveloping, and hence cluster tilting.

Then the projection functor $\sD \rightarrow \sD/\cU$ induces a
bijection between the cluster tilting subcategories of $\sD$ and the
support tilting subcategories of $\sD/\cU$ which are equal to their
image under $S\Sigma^{-2}$.
\end{Proposition}

\begin{proof}
The operation $\cV \stackrel{\pi}{\rightarrow} \underline{\cV}$
induced by the projection functor sends full sub\-ca\-te\-go\-ri\-es
of $\sD$ to full subcategories of $\sD/\cU$.  By Theorem
\ref{thm:VtoW}, it sends cluster tilting subcategories to support
tilting subcategories.  Cluster tilting subcategories are equal to
their image under $S\Sigma^{-2}$ by \cite[thm.\ 4.7.3]{KoenigZhu}, so
the support tilting subcategories arising from this are too.

The operation $\underline{\cW} \stackrel{\lambda}{\rightarrow} \cX$ of
Remark \ref{rmk:WtoV} sends full subcategories of $\sD/\cU$ to full
subcategories of $\sD$.  By Theorem \ref{thm:WtoV} and the assumption
of the present proposition, it sends support tilting subcategories
which are equal to their image under $S\Sigma^{-2}$ to cluster tilting
subcategories.

Let $\cV$ be cluster tilting in $\sD$.  Then $\cV$ and
$\lambda\pi(\cV)$ are both liftings of $\pi(\cV) = \underline{\cV}$ to
$\sD$, and they are both cluster tilting and so in particular maximal
$1$-orthogonal.  Hence $\lambda\pi(\cV) = \cV$ by Theorem
\ref{thm:WtoV}.

Let $\underline{\cW}$ be support tilting in $\sD/\cU$ with
$S\Sigma^{-2}\underline{\cW} = \underline{\cW}$.  Then $\cX =
\lambda(\underline{\cW})$ is a lifting of $\underline{\cW}$ to $\sD$,
that is, $\pi\lambda(\underline{\cW}) = \underline{\cW}$.

This shows that $\pi$ and $\lambda$ are mutually inverse maps between
the set of cluster tilting subcategories of $\sD$ and the set of
support tilting sub\-ca\-te\-go\-ri\-es of $\sD/\cU$ which are equal
to their image under $S\Sigma^{-2}$, and the proposition follows.
\end{proof}

\begin{Remark}
The situation of the proposition occurs in practice, as we will see in
some of the examples of the next section.  It would be interesting to
find a simple criterion which guarantees that we are in this situation.
\end{Remark}


\section{Examples}
\label{sec:examples}

\subsection{Cluster categories}
\label{subsec:cluster}

Let $Q$ be a finite quiver without loops or cycles, let $\sD$ be
the cluster category of type $Q$ over $k$, and consider the cluster
tilting subcategory $\cU = \add kQ$; cf.\ \cite{BMRRT}.

The conditions of Setup \ref{set:blanket} hold by \cite[sec.\
1]{BMRRT} and \cite[thm.\ 3.3(b)]{BMRRT}.

The quotient category $\sD/\cU$ is equivalent to $\mod\, kQ$, as
follows from the theory of \cite{BMR2}.  In particular, $\sD/\cU$ has
finite global dimension and all its objects have finite length.  Since
$\sD$ is $2$-Calabi-Yau as follows from \cite[sec.\ 1]{BMRRT}, the
functor $S\Sigma^{-2}$ is equivalent to the identity.

We claim that we are in the situation of Proposition \ref{pro:WtoV}.
To see this, we must consider a support tilting subcategory
$\underline{\cW}$ of $\sD/\cU \simeq \mod\, kQ$ and show that the
subcategory $\cX$ of Remark \ref{rmk:WtoV} and Theorem \ref{thm:WtoV}
is precovering and preenveloping.  But $\underline{\cW}$ is, in
particular, a partial tilting subcategory so contains only finitely
many isomorphism classes of indecomposable objects; see \cite[lem.\
VI.2.4 and cor.\ VI.4.4]{ASS}.  Since $\cU$ also contains only
finitely many isomorphism classes of in\-de\-com\-po\-sa\-ble objects,
the same is true for $\cX$ which is hence precovering and
preenveloping.

Proposition \ref{pro:WtoV} therefore says that the projection functor
$\sD \rightarrow \sD/\cU$ induces a bijection between the cluster
tilting subcategories of $\sD$ and the support tilting subcategories
of $\sD/\cU$.

Finally, observe that the cluster tilting subcategories of $\sD$
contain only finitely many isomorphism classes of indecomposable
objects, as follows from \cite[thm.\ 3.3]{BMRRT}, so the cluster
tilting subcategories are in bijection with the isomorphism classes of
basic cluster tilting objects of $\sD$.  Likewise, as mentioned, the
support tilting subcategories of $\sD/\cU$ contain only finitely many
isomorphism classes of indecomposable objects, so the support tilting
subcategories are in bijection with the isomorphism classes of basic
support tilting objects of $\sD/\cU$.

Hence the projection functor $\sD \rightarrow \sD/\cU$ induces a
bijection between the isomorphism classes of basic cluster tilting
objects of the cluster category $\sD$ and the isomorphism classes of
basic support tilting objects in $\sD/\cU \simeq \mod\, kQ$.

This is precisely the bijection of Ingalls and Thomas from Theorem A
of the Introduction.

\subsection{Derived categories}

Let $Q$ be a finite quiver without loops and cycles and set $\sD$
equal to $\Df(kQ)$, the finite derived category of the path algebra
$kQ$.  Consider $kQ$ itself as an object of $\sD$ and set $\cU$ equal
to $\add$ of the orbit of $kQ$ under $S\Sigma^{-2}$; cf.\
\cite[4.5.2]{KoenigZhu}. 

The conditions of Setup \ref{set:blanket} are satisfied: The $\Hom$
spaces of $\sD$ are finite dimensional by an explicit computation with
projective resolutions.  Idempotents in $\sD$ split because, by
\cite[prop.\ 3.2]{BN}, they do so in $\sD(kQ)$, the derived category
of all complexes.  And there is a Serre functor by \cite[3.6]{Happel}
and \cite[thm.\ I.2.4]{RVdB}.

Consider the module category $\mod\, kQ$.  Its Auslander-Reiten quiver
(AR quiver) $\Gamma$ typically consists of a preprojective component
of the form $\BN Q$, a regular component, and a preinjective component
which is the mirror image of the preprojective component.  The AR
quiver of $\sD$ is obtained by taking a countable number of copies of
$\Gamma$ and gluing them together, preinjective components to
preprojective components; cf.\ \cite{Happel}.  It typically looks as
follows, where the zig zags indicate the subcategory $\cU$.
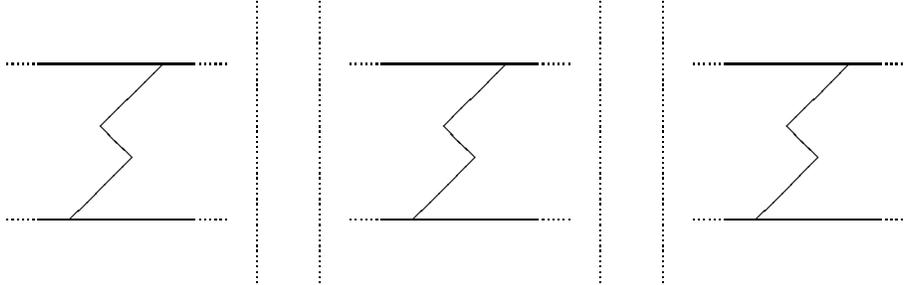
\[
  \xymatrix @-1.6pc @! {
    & & & & & & & & *{} \ar@{.}[ddddddddd] & & *{} \ar@{.}[ddddddddd] & & & & & & & & & *{} \ar@{.}[ddddddddd] & & *{} \ar@{.}[ddddddddd] & & & & & & & & \\
    & & & & & & & & & & & & & & & & & & & & & & & & & & & & & \\
    *{} \ar@{.}[r] & *{} \ar@{-}[rrrrr] & & & & *{} & *{} \ar@{.}[r] & *{} &&&&*{} \ar@{.}[r] & *{} \ar@{-}[rrrrr] & & & & *{} & *{} \ar@{.}[r] & *{} &&&&*{} \ar@{.}[r] & *{} \ar@{-}[rrrrr] & & & & *{} & *{} \ar@{.}[r] & *{} \\
    & & & & & & & & & & & & & & & & & & & & & & & & & & & & & \\
    & & & *{} \ar@{-}[uurr] & & & & & & & & & & & *{} \ar@{-}[uurr] & & & & & & & & & & & *{} \ar@{-}[uurr] & & & & \\
    & & & & *{} \ar@{-}[ul] & & & & & & & & & & & *{} \ar@{-}[ul] & & & & & & & & & & & *{} \ar@{-}[ul] & & & \\
    & & & & & & & & & & & & & & & & & & & & & & & & & & & & & \\
    *{} \ar@{.}[r] & *{} \ar@{-}[rrrrr] & *{} \ar@{-}[uurr] & & & & *{} \ar@{.}[r] & *{} &&&&*{} \ar@{.}[r] & *{} \ar@{-}[rrrrr] & *{} \ar@{-}[uurr] & & & & *{} \ar@{.}[r] & *{} &&&&*{} \ar@{.}[r] & *{} \ar@{-}[rrrrr] & *{} \ar@{-}[uurr] & & & & *{} \ar@{.}[r] & *{} \\
    & & & & & & & & & & & & & & & & & & & & & & & & & & & & & \\
    & & & & & & & & *{} & & *{} & & & & & & & & & *{} & & *{} & & & & & & & & \\
                     }
\]
The abelian quotient category $\sD/\cU \simeq \mod\, \cU$ is the
direct sum of countably many copies of $\mod\, kQ$, so it is clear
that $\sD/\cU$ has finite global dimension and that each of its
objects has finite length.

Note that in the AR quiver of $\sD$, the copies of $\Gamma$ which are
glued to obtain the quiver do not correspond to the copies of $\mod\,
kQ$ whose direct sum is $\sD/\cU$.  The former overlap with the
vertices corresponding to $\cU$, the latter correspond to their
complement. 


We claim that we are again in the situation of Proposition
\ref{pro:WtoV}, so the projection functor induces a bijection between
the cluster tilting subcategories of $\sD$ and the support tilting
subcategories of $\sD/\cU$ which are equal to their image under
$S\Sigma^{-2}$.

To see this, we must let $\underline{\cW}$ be a support tilting
subcategory of $\sD/\cU$ with $S\Sigma^{-2}\underline{\cW} =
\underline{\cW}$ and show that the lifted subcategory $\cX$ of Remark
\ref{rmk:WtoV} and Theorem \ref{thm:WtoV} is precovering and
preenveloping.  However, when $\underline{\cW}$ is support tilting
then its intersection with each copy of $\mod\, kQ$ inside $\sD/\cU$
is a partial tilting subcategory, and so only contains finitely many
isomorphism classes of indecomposable objects; cf.\ Section
\ref{subsec:cluster}.  This easily implies that $\cW$ only contains
finitely many isomorphism classes corresponding to vertices in each of
the copies of $\Gamma$ which are glued to form the AR quiver of $\sD$.
As the same is the case for $\cU$, it follows that it also holds for
$\cX$.

However, if $d$ is an indecomposable object of $\sD$, then the vertex
of $d$ sits in one of the copies of $\Gamma$.  The only indecomposable
objects of $\sD$ which have non-zero morphisms to and from $d$ are the
ones corresponding to vertices in that copy of $\Gamma$ and the two
neighbouring copies.  But this means that only finitely many
isomorphism classes of indecomposable objects from $\cX$ have non-zero
morphisms to and from $d$ whence $\cX$ is precovering and
preenveloping.

\subsection{A category of type $A_{\infty}$}

Let $R = k[X]$ be the polynomial algebra and view $R$ as a DG algebra
with zero differential and $X$ placed in homological degree $1$.  Let
$\sD$ be $\Df(R)$, the derived category of DG $R$-modules with finite
dimensional homology over $k$.

This category of type $A_{\infty}$ was studied in \cite{HJ} where it
was shown to exhibit cluster behaviour.  In particular, it was shown
that its maximal $1$-orthogonal subcategories are in bijection with
the set of maximal configurations of non-crossing arcs connecting
non-neighbouring integers.  It was also shown that not all maximal
$1$-orthogonal subcategories are cluster tilting; indeed, a precise
criterion was given to decide whether a maximal configuration of arcs
determines a cluster tilting subcategory.

The category $\sD$ satisfies Setup \ref{set:blanket} by \cite{HJ}.  It
is $2$-Calabi-Yau so the functor $S\Sigma^{-2}$ is equivalent to the
identity.  Its AR quiver is $\BZ A_{\infty}$.  Let $\cU$ be $\add$ of
infinitely many indecomposable objects, the first of which are
indicated by solid dots in the following sketch of the AR quiver.
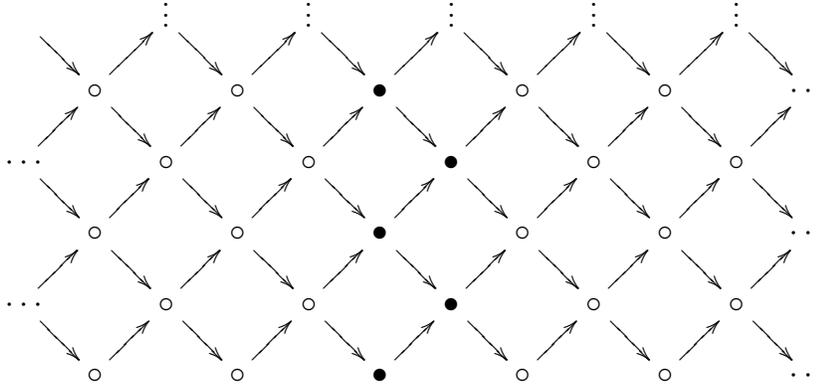
\[
  \xymatrix @-1.6pc @! {
    \rule{2ex}{0ex} \ar[dr] & & \vdots \ar[dr] & & \vdots \ar[dr] & & \vdots \ar[dr] & & \vdots \ar[dr] & & \vdots \ar[dr] & & \rule{2ex}{0ex} \\
    & \circ \ar[ur] \ar[dr] & & \circ \ar[ur] \ar[dr] & & \bullet \ar[ur] \ar[dr] & & \circ \ar[ur] \ar[dr] & & \circ \ar[ur] \ar[dr] & & \cdots & \\
    \cdots \ar[ur]\ar[dr]& & \circ \ar[ur] \ar[dr] & & \circ \ar[ur] \ar[dr] & & \bullet \ar[ur] \ar[dr] & & \circ \ar[ur] \ar[dr] & & \circ \ar[ur] \ar[dr] \\
    & \circ \ar[ur] \ar[dr] & & \circ \ar[ur] \ar[dr] & & \bullet \ar[ur] \ar[dr] & & \circ \ar[ur] \ar[dr] & & \circ \ar[ur] \ar[dr] & & \cdots & \\
    \cdots \ar[ur]\ar[dr]& & \circ \ar[ur] \ar[dr] & & \circ \ar[ur] \ar[dr] & & \bullet \ar[ur] \ar[dr] & & \circ \ar[ur] \ar[dr] & & \circ \ar[ur] \ar[dr] \\
    & \circ \ar[ur] & & \circ \ar[ur] & & \bullet \ar[ur] & & \circ \ar[ur] & & \circ \ar[ur] & & \cdots & \\
               }
\]
It was shown in \cite{HJ} that $\cU$ is a cluster tilting subcategory
of $\sD$ and that $\cU$ is equivalent to the free category on its AR
quiver $Q$,
\[
  \xymatrix{
    \bullet \ar[r] & \bullet & \bullet \ar[r] \ar[l] & \bullet & \bullet \ar[r] \ar[l] & \cdots.
           }
\]
Accordingly, $\sD/\cU \simeq \mod\, \cU$ is equivalent to $\rep\, Q$,
the category of finitely presented representations of $Q$, which is
hereditary by \cite[sec.\ II.1]{RVdB}.  Since $Q$ is locally finite,
each object of $\rep\, Q$ has finite length.

It follows that Theorems \ref{thm:VtoW} and \ref{thm:WtoV} both apply,
so cluster tilting sub\-ca\-te\-go\-ri\-es of $\sD$ project to support
tilting subcategories of $\sD/\cU$, and support tilting subcategories
of $\sD/\cU$ can be lifted uniquely to maximal $1$-orthogonal
subcategories of $\sD$.

In particular, any configuration of arcs which determines a cluster
til\-ting subcategory of $\sD$ also gives rise to a support tilting
subcategory of $\sD/\cU$, so we get an ample supply of such
subcategories.

We do not know whether Proposition \ref{pro:WtoV} applies to this
situation.  Support tilting subcategories of $\sD/\cU$ lift to maximal
$1$-orthogonal sub\-ca\-te\-go\-ri\-es of $\sD$, but not all such
subcategories are cluster tilting.  It would be interesting to
determine whether or not Proposition \ref{pro:WtoV} does apply.

\medskip
{\bf Acknowledgement. }
We thank Bill Crawley-Boevey, Colin Ingalls, and Hugh Thomas for
answering several questions on support tilting.  The se\-cond author
carried out part of the work during a visit to Osaka Prefecture
University in July 2008, and wishes to express his sincere gratitude
to his host, Kiriko Kato.

\end{document}